\date{13 April 2010}

\documentclass[12pt]{amsart}
\usepackage{latexsym,amsmath,amsfonts,amscd,amssymb}

\theoremstyle{plain}
\newtheorem{theorem}{Theorem}[section]
\newtheorem{corollary}[theorem]{Corollary}
\newtheorem{lemma}[theorem]{Lemma}
\newtheorem{proposition}[theorem]{Proposition}
\theoremstyle{definition}
\newtheorem{definition}[theorem]{Definition}
\newtheorem{remark}[theorem]{Remark}
\newtheorem{question}[theorem]{Question}
\numberwithin{equation}{section}

\renewcommand{\a}{\alpha}
\renewcommand{\b}{\beta}

\newcommand{\bd}{\partial}

\newcommand{\cA}{\mathcal{A}}

\newcommand{\cC}{\mathcal{C}}

\newcommand{\cH}{\mathcal{H}}

\newcommand{\cK}{\mathcal{K}}

\newcommand{\cO}{\mathcal{O}}

\newcommand{\QQ}{\mathbb{Q}}

\newcommand{\RR}{\mathbb{R}}
\newcommand{\CC}{\mathbb{C}}
\newcommand{\HH}{\mathbb{H}}

\newcommand{\TT}{\mathbb{T}}
\newcommand{\ZZ}{\mathbb{Z}}

\newcommand{\inc}{\hookrightarrow}

\newcommand{\too}{\longrightarrow}
\newcommand{\x}{\times}
\newcommand{\ox}{\otimes}

\newcommand{\la}{\langle}
\newcommand{\ra}{\rangle}

\newcommand{\Tr}{\mathrm{Tr}}

\newcommand{\Sym}{\mathrm{Sym}}

\newcommand{\vol}{\mathrm{vol}}
\newcommand{\Map}{\mathrm{Map}}

 \DeclareMathOperator{\im}{im}

\DeclareMathOperator{\Hom}{Hom} 
\DeclareMathOperator{\End}{End} 
\DeclareMathOperator{\GL}{GL}
\DeclareMathOperator{\supp}{supp}

\title{Hodge theory for Riemannian solenoids}

\subjclass[2000]{Primary: Secondary: }

\author{Vicente Mu\~noz}
  \address{Facultad de Matem\'{a}ticas \\ Universidad Complutense
  de Madrid \\ Plaza Ciencias 3  \\ 28040 Madrid \\ Spain}
  \email{vicente.munoz@mat.ucm.es}

\author[R. P\'{e}rez Marco]{Ricardo P\'{e}rez Marco}
\address{CNRS, LAGA UMR 7539, Universit\'e Paris XIII \\
99, Avenue J.-B. Cl\'ement, 93430-Villetaneuse, France}
\email{ricardo@math.univ-paris13.fr}

\thanks{Partially supported through grant MEC
(Spain) MTM2007-63582}

\begin{document}
\maketitle

\begin{abstract}
  A measured solenoid is a compact laminated space endowed with a
  transversal measure.
  The De Rham $L^2$-cohomology
  of the solenoid is defined by using differential forms
  which are smooth in the leafwise directions and $L^2$
  in the transversal direction.
  We develop the theory of harmonic forms for Riemannian
  measured solenoids, and prove that this computes the De
  Rham $L^2$-cohomology of the solenoid. This implies in particular
  a Poincar\'{e} duality result.
\end{abstract}

{\hfill {\emph{Dedicated to the Memory of the 100th Anniversary of S.\ M.\ Ulam}}}

\section{Introduction} \label{sec:introduction}

In this article we continue the geometric study of solenoids initiated in \cite{MPM1}. Solenoids
are compact topological spaces which are locally modeled by the
product of a $k$-dimensional ball by some transversal space and admit a
transversal measure invariant by holonomy.

First we review how the De Rham cohomology theory, fundamental
classes, and singular cohomology do extend to solenoids and preserve most of the functorial
properties. The theory of bundles, connections and Chern classes
also goes through.

For a solenoid $S$ endowed with a transversal measure $\mu$, we
have a well defined De Rham $L^2$-cohomology theory, which is suited to implement
classical techniques from Harmonic and Functional Analysis. The De Rham $L^2$-cohomology is
defined by using forms which are smooth in the leaf-wise directions, and are $L^2$-integrable
with respect to $\mu$ in the transversal direction. We introduce the \emph{reduced} De Rham
$L^2$-cohomology $\bar{H}_{DR}^*(S_\mu)$ as the quotient the De Rham $L^2$-cohomology $H_{DR}^*(S_\mu)$
with the closure of $\{0\}$ (making it a Hausdorff topological space). See \cite{MoSch} for
these notions.

Most of the formal theory of pseudo-differential operators can be carried out for spaces of
sections of bundles which are $L^2$-transversally. In order to extend
classical Hodge theory as developed in \cite{Wells}, we define Sobolev spaces of sections
using the $L^2$ transversal structure. Sobolev regularity lemma
holds in this general setting but Rellich compact embedding does not in general. With these
tools at hand we can develop harmonic theory for solenoids.

We obtain a Hodge theorem giving an isomorphism of the reduced De Rham cohomology with the
space of harmonic forms. We prove the main results:

\begin{theorem}
Let $S_\mu$ be a compact oriented solenoid endowed with a transversal measure $\mu$.
Let $\cK^p(S_\mu)$ be the space of $p$-forms which are harmonic in the
leaf-wise directions and $L^2$-transversally (with respect to $\mu$). Then there is an
isomorphism:
  $$
  \bar{H}_{DR}^p(S_\mu) \cong \cK^p(S_\mu) \, .
  $$
\end{theorem}

\begin{corollary}
The $*$-Hodge operator gives an isomorphism (Poincar\'e duality) $*: \bar{H}_{DR}^p(S_\mu) \to
 \bar{H}_{DR}^{k-p}(S_\mu)$.
\end{corollary}

Contrary to the
classical situation for compact manifolds where Rellich theorem holds, here the spaces of harmonic
forms are not necessarily
finite dimensional. Actually, the situation for solenoids is more similar to that of
Hodge theory for $L^2$-forms for complete non-compact
manifolds (see e.g.\ \cite{complete}).
However, when the transversal measure $\mu$ is ergodic (see definition \ref{def:transversal-ergodicity}), we may expect finite-dimensionality (see question \ref{conj}).
We discuss at the end the simple example of Kronecker foliations on the torus.

\section{Solenoids} \label{sec:solenoids}

We review the basic notions on solenoids introduced in \cite{MPM1}.

\begin{definition}\label{def:k-solenoid} Let $k\geq 0$, $r\geq 1$, $s\geq 0$ with $r\geq s$. A
$k$-solenoid of class $C^{r,s}$ is a compact Hausdorff space endowed with an atlas
of flow-boxes $\cA=\{ (U_i,\varphi_i)\}$,
 $$
 \varphi_i:U_i\to D^k\x K(U_i)\, ,
 $$
where $D^k$ is the $k$-dimensional open ball, and $K(U_i)\subset \RR^l$ is an open set, the transversal
of the flow-box. The changes of charts $\varphi_{ij}=\varphi_i\circ
\varphi_j^{-1}$ are of the form $\varphi_{ij}(x,y)=(X(x,y), Y(y))$,
where $X(x,y)$ is of class $C^{r,s}$ and $Y(y)$ is of class $C^s$.
\end{definition}

We shall always use \emph{good} flow-boxes. By this, we mean a flow-box $U=D^k\x K(U)$ whose closure
is contained in another flow-box $V=D^k \x K(V)$. Therefore $\overline{K(U)} \subset K(V)$ is
compact.

Let $S$ be a $k$-solenoid, and $U\cong D^k \x K(U)$ be a flow-box for $S$. The sets
$L_y= D^k\x \{y\}$ are called the (local) leaves of the flow-box. A leaf $l\subset S$ of the
solenoid is a connected $k$-dimensional manifold whose intersection with any flow-box
is a collection of local leaves. The solenoid is oriented if the leaves are oriented
(in a transversally continuous way).

A transversal for $S$ is a subset $T$ which is a finite union of transversals of flow-boxes.
Given two local transversals $T_1$ and $T_2$ and
a path contained in a leaf from a point of $T_1$ to a point of $T_2$,
there is a well-defined germ of holonomy map at this point, $h$, from $T_1$ to $T_2$.

\begin{definition} \label{def:transversal-measure}
Let $S$ be a $k$-solenoid. A transversal measure $\mu=(\mu_T)$ for
$S$ associates to any local transversal $T$ a locally finite measure
$\mu_T$ supported on $T$, which are invariant by the holonomy,
i.e. if $h$ is a germ of holonomy map, then
$h_* \mu_{T_1}= \mu_{T_2}$.
\end{definition}

We denote by $S_\mu$ a $k$-solenoid $S$ endowed with a transversal
measure $\mu=(\mu_T)$. We refer to $S_\mu$ as a measured solenoid.
Observe that for any transversal measure $\mu=(\mu_T)$ the scalar
multiple $c\, \mu=(c \, \mu_T)$, where $c>0$, is also a transversal
measure. Notice that there is no natural scalar normalization of
transversal measures.

\begin{definition} \label{def:transversal-ergodicity}
A measured solenoid $S_\mu$ is \emph{ergodic} if for any transveral
$T$, and any subset $A\subset T$ invariant by the holonomy, either
$\mu_T(A)=0$ or $\mu_T(T-A)=0$.

A solenoid $S$ is
\emph{uniquely ergodic} if it has a unique (up to scalars)
transversal measure $\mu$ and its support is the whole of $S$.
\end{definition}

A Riemannian solenoid is a solenoid
endowed with a Riemannian metric in the tangent spaces of the leaves, and
with smoothness of class $C^{r,s}$. Note that a Riemannian metric defines
a volume form in each leaf. A {\it daval} measure $\nu$
on $S$ (\cite{MPM1}) is a finite Borel measure on
the solenoid which in any flow-box $U=D^k\times K(U)$, it decomposes as volume along leaves,
$\nu= \vol_{D^k}\times \mu_{K(U)}$. Such a
measure defines a tranversal measure and moreover there is a one-to-one correspondence
between daval measures and transversal measures. In particular, if a Riemannian solenoid
is uniquely ergodic, then there is a
\emph{unique} daval measure with total mass $1$.

\bigskip

Now let $M$ be a smooth manifold of dimension $n$ (and class $C^r$). An immersion of a $k$-solenoid
$S$ into $M$, with $k<n$, is a smooth map $f:S\to M$ (of class $C^{r,s}$) such that the differential
restricted to the tangent spaces of leaves has rank $k$ at every
point of $S$. If $M$ is a Riemannian manifold, this endows $S$ with the pull-back
Riemannian structure.

Denote by $\cC_k(M)$
the space of compactly supported currents of dimension $k$
on $M$. We have the following definition.

\begin{definition}\label{def:Ruelle-Sullivan}
Let $S_\mu$ be an oriented measured $k$-solenoid. An immersion
$f:S\to M$
defines a generalized Ruelle-Sullivan current $(S_\mu,f)\in \cC_k(M)$ as follows.
Let $\{U_i\}$ be a finite covering of $S$ by flow-boxes, and a partition of unity
$\{\rho_i\}$ associated to it. For $\omega\in \Omega^k(M)$, we define
 $$
 \la (S_\mu,f),\omega \ra=\sum_i \int_{K(U_i)} \left ( \int_{L_y}
 \rho_i f^* \omega \right ) \ d\mu_{K(U_i)} (y) \, ,
 $$
where $L_y$ denotes the horizontal disk of the flow-box.
\end{definition}

In \cite{MPM1} it is proved that $(S_\mu,f)$ is a closed current. Therefore, it defines
a real homology class
 \begin{equation}\label{eqn:RS}
 [S_\mu, f]\in H_k(M,\RR).
 \end{equation}
Ruelle and Sullivan defined
in \cite{RS} this notion for the restricted class of solenoids embedded in $M$.
In \cite{MPM3}, it is proved that if $a\in H_k(M,\RR)$ is any real homology class,
then there is an immersed oriented $k$-solenoid $f:S_\mu \to M$ such that $a=[S_\mu, f]$.
Moreover, $S$ can be chosen to be uniquely ergodic and with holonomy generated by a single
map. In \cite{MPM5}, we prove that the set of currents $(S_\mu,f)$ for immersed oriented
uniquely ergodic $k$-solenoids with $a=[S_\mu, f]$ is actually dense in the space of
closed currents $\a\in \cC_k(M)$ representing $a$.

\section{Cohomology of solenoids}\label{sec:cohomology}

In general, and for the remainder of the article, we shall consider solenoids
of class $C^{\infty,0}$.

\subsection{De Rham cohomology} \label{subsec:DrRham}

Let $S$ be a solenoid (here, we allow $S$ to be a \emph{non-compact} solenoid).
The space of $p$-forms $\Omega^p(S)$ consist of $p$-forms on leaves with function coefficients
that are smooth on leaves and partial derivatives of all orders continuous transversally.
Using the differential $d$ in the leaf-wise directions, we obtain the De Rham differential
complex $(\Omega^*(S),d)$. The De Rham cohomology groups
of the solenoid are defined as the quotients
  \begin{equation}\label{eqn:DR}
  H^p_{DR}(S):= \frac{\ker (d:\Omega^p(S)\to \Omega^{p+1}(S))}{\im
  (d: \Omega^{p-1}(S)\to \Omega^{p}(S))}\, .
  \end{equation}
We can also consider the spaces $\Omega^p_m(S)$ of differential forms with function coefficients
that are smooth on leaves and measurable transversally
(then the partial derivatives are automatically measurable transversally). Then define in the same way
the De Rham measurable cohomology groups $H^p_{DRm}(S)$ using the complex $(\Omega^*_m(S),d)$.
Note the natural map $H^p_{DR}(S) \to H^p_{DRm}(S)$.

\begin{proposition}
Let $\underline{\RR}_c$ and $\underline{\RR}_m$ be respectively the sheaf of functions which are locally
constant on leaves and transversally continuous, resp. measurable.
Then we have isomorphisms
 $$
  H_{DR}^p(S) \cong H^p(S,\underline{\RR}_c)\, ,
 $$
and
 $$
  H_{DRm}^p(S) \cong H^p(S,\underline{\RR}_m)\, .
 $$
\end{proposition}

\begin{proof}
The proofs are similar. We prove the first isomorphism. It follows from the existence of the sheaf resolution
 $$
 \underline{\RR}_c \to \Omega^0 \stackrel{d}{\to} \Omega^1 \stackrel{d}{\to} \ldots
 $$
The exactness of this complex of sheaves is a Poincar\'{e} lemma: for a small open set $U=D^k\times K(U)$,
the complex $0\to \underline{\RR}_c(U) \to \Omega^0(U) \to \Omega^1(U) {\to} \ldots$
is exact. If $d_x \alpha(x,y)=0$ then $\alpha(x,y)=d_x\beta(x,y)$, for a collection of
forms $\beta(x,y)$, $y\in K(U)$. We can choose $\beta(x,y)$ to depend continuously (or measurably for the
proof of the second isomorphism) on
$y$, as can be seen by the usual construction. Note that clearly if $f(x,y)$ is a function
with $df=0$ then $f(x,y)$ is locally constant on $x$.
\end{proof}

\begin{remark}
The spaces $\Omega^p(S)$ are topological vector spaces. Therefore the De Rham cohomology
(\ref{eqn:DR}) inherits a natural topology. In general, these spaces are infinite dimensional (even for
compact solenoids). In some references, it is customary to take the closure of the spaces
$\im d$ in definition (\ref{eqn:DR}), obtaining the \emph{reduced De Rham cohomology groups}
 $$
 \bar H^p_{DR}(S)=\frac{\ker d|_{\Omega^p}}{\, \overline{ \im d|_{\Omega^{p1}}}\, }\, .
  $$
This is equivalent to quotienting $H^p_{DR}(S)$ by $\overline{\{0\}}$, obtaining thus
Hausdorff vector spaces. 
\end{remark}

We shall list some basic properties of the De Rham cohomology:
\begin{enumerate}
 \item Functoriality. Let $S_1$, $S_2$ be two solenoids. A smooth map
 $f:S_1\to S_2$ is a map sending leaves to leaves and transversally continuous.
 $f$ defines a map on De Rham cohomologies, $f^*: H_{DR}^p(S_2) \to H_{DR}^p(S_1)$,
 by $f^*[\omega]=[f^*\omega]$. This applies in particular to an immersion of a solenoid
 into a smooth manifold $f:S\to M$, or to the inclusion of a leaf $i:l\to S$.

 \item Mayer-Vietoris sequence. Let $U,V$ be two open subsets of a
solenoid $S$. There is a short exact
sequence of complexes:
$\Omega^\bullet (U\cup V) \to\Omega^\bullet (U)\oplus \Omega^\bullet (V) \to
 \Omega^\bullet (U\cap V)$. The only non-trivial point is the surjectivity of
 the last map, but if follows from the existence of a partition of unity
$\{\rho_U,\rho_V\}$ subordinated to $\{U,V\}$: any
$\omega\in \Omega^\bullet (U\cap V)$
is the image of $(\rho_V \omega, -\rho_U \omega)$. Taking the associated long
exact sequence, we get the
Mayer-Vietoris exact sequence
 $$
 \ldots \to H^p_{DR}(U\cup V)\to H^p_{DR}(U)\oplus H^p_{DR}(V) \to H^p_{DR}(U\cap V) \to
  H^{p+1}_{DR}(U\cup V)\to \ldots
 $$

 \item Homotopy. A homotopy between two maps
 $f_0,f_1:S_1\to S_2$ is a map $F: S_1\times [0,1] \to S_2$ (where
 $S_1\times [0,1]$ is given the solenoid structure with leaves $l\times [0,1]$,
 for $l\subset S_1$ a leaf of $S_1$) such that $F(x,0)=f_0(x)$ and $F(x,1)=f_1(x)$.
 We say that the maps $f_0,f_1$ are homotopic, written $f_0\sim f_1$.
 In this case $f_0^* =f_1^*:H_{DR}^p(S_2)\to H_{DR}^p(S_1)$.

 We prove this as follows: factor $f_0=F\circ i_0$, and $f_1=F\circ i_1$, where $i_t:S_1
 \to S_1\times [0,1]$ is given as $i_t(x)=(x,t)$. Then
 $f_0^*=f_1^*$ follows from $i_0^*=i_1^*$. Let us check this. Consider
 $p:S_1\times [0,1]\to S_1$. Then $p\circ i_t=Id$. Let us see that
 $i_t^*: H^p_{DR}(S\x I) \to H^p_{DR}(S)$ is an isomorphism inverse to $p^*$.
 It is enough to see that $h^*=(i_t\circ p)^*: H^p_{DR}(S\x I)\to H^p_{DR}(S\x I)$
 is the identity, $h(x,s)=(x,t)$, $t\in I$ fixed.
 For a closed $p$-form $\omega=\omega_1(x,s)+
 \omega_2(x,s)\wedge ds$, note that
 $0=d\omega=d_x\omega_1 + (-1)^p\frac{\bd \omega_1}{\bd s} \wedge ds +
 d_x\omega_2 \wedge ds=0$ implies that $d_x\omega_1=0$ and
 $\frac{\bd\omega_1}{\bd s} = (-1)^{p+1} d_x\omega_2$.
 Now $p^*i^*_t\omega= \omega_1(x,t)$. Hence
 $$
 \begin{aligned}
  (Id-p^*i_t^*) \omega &=  \omega_1(x,s)- \omega_1(x,t) +
  \omega_2(x,s)\wedge ds \\ &= \int_t^s \frac{\bd\omega_1}{\bd s}(x,u) du +\omega_2(x,s)\wedge ds = \\
  &=  (-1)^{p+1} \left(\int_t^s d_x \omega_2(x,u) du + (-1)^{p-1} \omega_2(x,s)\wedge ds\right) \\
   &= (-1)^{p+1} d \left(\int_t^s \omega_2(x,u)du \right),
  \end{aligned}
   $$
  as required.

\item We say that two solenoids $S_1$, $S_2$ are of the same homotopy
type if there are maps $f:S_1\to S_2$, $g:S_2\to S_1$, such that
$f\circ g\sim Id_{S_2}$, $g\circ f\sim Id_{S_1}$. Then the De Rham cohomology
groups of $S_1$ and $S_2$ are
isomorphic.

\end{enumerate}

\subsection{Fundamental classes}
Let $S$ be an \emph{oriented} compact $k$-solenoid.
The De Rham cohomology groups do not depend on any measure of $S$.
If $\mu=(\mu_T)$ is a transversal measure, then the integral
$\int_{S_\mu}$ descends to cohomology giving a map \cite{MPM1}
  \begin{equation}\label{eqn:int}
  \int_{S_\mu} : H^k_{DR}(S) \to \RR\ .
  \end{equation}

We define the solenoidal homology as
 $$
 H_p(S,\underline\RR_c):=H^p(S,\underline{\RR}_c)^*=H^p_{DR}(S)^*.
 $$
Then the map (\ref{eqn:int}) defines a homology class $[S_\mu] \in
H^k(S,\underline\RR_c)^*= H_k(S,\underline\RR_c)$. We shall call this
element the \emph{fundamental class} of $S_\mu$.

Any map $f:S_1\to S_2$ defines a map $f^*:H^p_{DR}(S_2)\to
H^p_{DR}(S_1)$ and hence, by dualizing, a map $f_*:H_p(S_1,\underline{\RR}_c) \to
H_p(S_2,\underline{\RR}_c)$. Applying this to an immersion $f:S_\mu\to M$ of an oriented,
measured, compact solenoid into a smooth manifold, then we have the equality
  $$
  f_*[S_\mu]= [S_\mu,f]\ ,
  $$
with the generalized Ruelle-Sullivan class defined in (\ref{eqn:RS}).

Note that if $S$ has a dense leaf (in particular when $S$ it is minimal, i.e. all leaves are dense), then
$H_0(S,\underline{\RR}_c)=\RR$. On the other hand, the dimension of the top degree homology counts
the number of mutually singular tranverse measures on $S$.

\begin{theorem} \label{thm:top}
 Let $S$ be a compact, oriented $k$-solenoid. Then $H_k(S,\underline\RR_c)$ is
 isomorphic to the real vector space generated by all transversal measures.
\end{theorem}

\begin{proof}
  There is a well-defined linear map
   $$
   \Psi:\mu \mapsto [S_\mu] \in H_k(S,\underline{\RR}_c)
   $$
  which sends each signed transversal measure (with finite total mass in each transversal)
  to the associated fundamental class: Decompose $\mu=\mu_+-\mu_-$ and integrate with respect to each measure.
  We prove first that $\Psi$ is a bijection.

  If $\Psi(\mu)=0$ then
   $$
   \int_{S_\mu} \omega =0
   $$
  for all $k$-forms $\omega\in\Omega^k(S)$. Let $U=D^k \times K(U)$
  be a flow-box and let $\omega$ be a compactly supported
  $k$-form with integral $1$ on $D^k$. For any function $\phi$ with
  $\supp \phi \subset K(U)$, we have
   $$
   0=\int_{S_\mu} \phi\, \omega=\int_{K(U)} \phi \left(\int_{L_y}
   \omega\right) d\mu_{K(U)}(y) =\int_{K(U)} \phi \, d\mu_{K(U)}(y)\, .
   $$
   Therefore $\mu_{K(U)}=0$. So $\mu=0$.

  We prove that $\Psi$ is onto. Let $C\in H_k(S,\underline{\RR}_c)$ be given. By definition
  $C:H^k_{DR}(S)\to \RR$, so for any $k$-form $\omega$, we have
  $C(\omega)\in\RR$ and $C(d\eta)=0$, for $\eta\in
  \Omega^{k-1}(S)$. 
  Let $U=D^k \times K(U)$ be a flow-box, and fix a compactly supported
  $k$-form $\omega$ with integral $1$ on $D^k$. Let
  $\phi$ be a continuous function with  $\supp \phi \subset K(U)$
  and let $\tilde \phi$  be the corresponding function
  on the flow box $U=D^k\times K(U)$ constant on leaves. Now the map
  $$
  \phi \mapsto C(\tilde \phi\,  \omega )\ ,
   $$
  is a continuous linear functional: if $\phi_n\to \phi$ in $C^0$ then  $\tilde \phi_n\, \omega\to
  \tilde \phi\, \omega$ in $\Omega^k(S)$, so $C(\tilde \phi_n\, \omega) \to
  C(\tilde \phi\, \omega)$. Therefore by Riesz representation theorem it is represented by
  a signed measure $\mu_{K(U)}$ with finite total mass,
  $$
  C(\tilde \phi \, \omega)= \int_{K(U)} \phi \, d\mu_{K(U)} \ .
  $$
  The measure $\mu_{K(U)}$ does not depend on the choice of $\omega$.
  Taking another $\omega'$ we
  have $\omega'-\omega=d\eta$, with $\eta$ compactly supported in
  $D^k$. Hence $\tilde \phi\,\omega'-\tilde \phi\,\omega = d(\tilde \phi\,\eta)$ and
  $C(\tilde \phi\,\omega')=C(\tilde \phi\,\omega)$.

  Also the constructed measure does not depend on the coordinates of the flow-box.
   Taking a change of chart $\Phi: D^k\times K(U)\to
  D^k\times K(U)$, we have that $C(\Phi^*(\tilde \phi\,\omega))=C(\tilde \phi\,\omega)$
  since $\Phi^*(\tilde \phi\,\omega)- \tilde \phi\,\omega =d\eta$ (both are
  compactly supported forms with leaf-wise integral $1$).
  Finally, the $(\mu_T)$ are invariant by the holonomy. We only need to check the
  invariance by local holonomy, and this follows from
  the previous remark.
\end{proof}

\begin{remark}
There is no Poincar\'e duality for $H_{DR}^*(S)$ in general.
Moreover these spaces may be
infinite dimensional (even for uniquely ergodic solenoids):
if $S$ is an $n$-torus foliated by irrational lines, then
$H_{DR}^1(S)$ can be infinite-dimensional as the example 
in section \ref{sec:example} shows.
\end{remark}

\subsection{Singular cohomology}

We consider the space $\Map(I^n,S)$ of continuous maps $T:I^n \to S$ mapping
into a leaf, and endow it with the uniform convergence topology.
The degenerate maps (see \cite{Ma}) form a closed subspace, therefore the quotient,
$\Map'(I^n,S)$, has a natural quotient topology. The space of
singular chains $C_n(S)$ is the free abelian group generated by
$\Map'(I^n,S)$. There is a natural boundary map $\bd: C_n(S)\to C_{n-1}(S)$.

Let $G$ be any topological abelian group.
Define the cochains $C^n(S,G)=\Hom_{cont}(C_n(S),G)$ as the
continous homomorphisms. That is, $\varphi:C_n(S)\to G$ such that
if $T_k:I^n\to S$ are maps which converge to $T_o:I^n\to S$
in the uniform topology, then $\varphi(T_k)\to \varphi(T_o)$.
Define the differential $\delta:C^n(S)\to C^{n+1}(S)$ by
$\delta \varphi(T)= \varphi(\bd T)$. The solenoid
singular cohomology of $S$ with coefficients in $G$ is defined as:
  $$
  H^n(S,G):=\frac{\ker(\delta: {C}^n(S,G)\to {C}^{n+1}(S,G))}{\im
  (\delta: {C}^{n-1}(S,G)\to {C}^{n}(S,G))}\, .
  $$

We have some basic properties:
\begin{enumerate}
 \item Functoriality. Let $f:S_1\to S_2$ be a solenoid map. Then
 there is a map $f_*:C_n(S_1)\to C_n(S_2)$, $f_*(T)=f\circ T$, and
 a map $f^*:C^n(S_2,G)\to C^n(S_1,G)$, $f^*(\varphi)=\varphi\circ f$.
 Clearly $f^* \delta=\delta f^*$, so the map descends to cohomology:
 $f^*:H^n(S_2,G)\to H^n(S_1,G)$.
 \item Homotopy. Suppose that $f,g:S_1\to S_2$ are two homotopic
solenoid maps. The usual construction yields a chain homotopy $H$ between
$f^*$ and $g^*$ (one only have to check that this map sends continuous
cochains into continuous cochains). Therefore $f^*=g^*
:H^n(S_2,G)\to H^n(S_1,G)$.
 \item If $S_1$, $S_2$ are of the same homotopy type, then $H^n(S_1,G)\cong H^n(S_2,G)$.
 \item If $U=D^k\times K(U)$ is a flow-box, then $U$ is of the same homotopy type than
 $\{*\}\x K(U)$. Therefore $H^n(U)=0$ for $n>0$, and $H^0(U)=\Map_{cont}(K(U),G)$. In particular,
this implies that
 $$
 \underline{\RR}_c \to C^0(-,\RR) \stackrel{\delta}{\to} C^1(-,\RR) \stackrel{\delta}{\to} \ldots
 $$
 is a resolution. Therefore there is an isomorphism $H^n(S,\RR)\cong H^n(S,\underline{\RR}_c)$.
\item Mayer-Vietoris. For two open sets $U,V$ with $S=U\cup V$,
define $C_n(S;U,V)$ as the subcomplex generated
by those singular chains completely contained in either $U$ or $V$. Define accordingly $C^n(S;U,V)$.
It is not difficult to see that the restriction $C^n(S,G) \to C^n(S,G;U,V)$ is chain homotopy
equivalence (by a process of subdivision of simplices, as in \cite{Ma}). Therefore the exact
sequence $0\to C^n(S,G;U,V) \to C^n(U,G)\oplus C^n(V,G)\to C^n(U\cap V, G)\to 0$ gives rise to a
long exact sequence:
 $$
 \ldots \to H^p(U\cup V,G)\to H^p(U,G)\oplus H^p(V,G) \to H^p(U\cap V,G) \to
  H^{p+1}(U\cup V,G)\to \ldots
 $$
\end{enumerate}

\section{De Rham $L^2$-cohomology} \label{sec:L2-cohomology}

Now consider a $k$-solenoid $S$ with a transversal measure $\mu$.
There is a notion of cohomology which takes into account the transversal measure
structure. For this, we work with forms which are $L^2$-transversal
relative to $\mu$.

\begin{definition}
A function $f$ is \emph{$L^2(\mu)$-transversally smooth} if in any (good) flow-box
$U=D^k\x K(U)$ all partial derivatives on the first variable exist and
are in $L^2(\mu_{K(U)})$, i.e. if we write $f$ as $f(x,y)$ then
for all $r\geq 0$,
 $$
 \int_{K(U)}  ||f( \cdot ,y)||_{C^r}^2 \,  d\mu_{K(U)}(y) <\infty \ .
 $$
\end{definition}

We consider the space of forms
 $$
 \Omega_{L^2(\mu)}^p(S)
 $$
which are $L^2(\mu)$-transversally smooth, i.e. locally these are forms
$\alpha= \sum f_I(x,y) dx_I$, where $f_I$ are $L^2(\mu)$-transversally smooth functions.
There is a well-defined differential along leaves $d:\Omega^p_{L^2(\mu)}(S) \to \Omega^{p+1}_{L^2(\mu)}(S)$
which defines the complex $(\Omega^*_{L^2(\mu)}(S), d)$. We define the
\emph{De Rham $L^2$-cohomology} vector space as the quotients
  \begin{equation}\label{eqn:DRL2}
  H^p_{DR}(S_\mu):= \frac{\ker (d:\Omega^p_{L^2(\mu)}(S)\to \Omega^{p+1}_{L^2(\mu)}(S))}{\im
  (d: \Omega^{p-1}_{L^2(\mu)}(S)\to \Omega^{p}_{L^2(\mu)}(S))}\, .
  \end{equation}
We also introduce the reduced De Rham $L^2$-cohomology:
  \begin{equation}\label{eqn:DRL2}
  \bar H^p_{DR}(S_\mu):= \frac{\ker d}{\, \overline{\im d}\, } \, .
  \end{equation}

Note that there are natural maps
 $$
 H^p_{DR}(S) \to H^p_{DR}(S_\mu) \to H^p_{DRm}(S) \, ,
 $$
since $C^{\infty,0}$-functions are $L^2(\mu )$-transversally smooth.
The integration map $\int_{S_\mu}$ is well-defined for forms in
$\Omega^k_{L^2(\mu)}$, since a $L^2(\mu)$-transversally smooth
$k$-form is automatically $L^1(\mu)$-transversal (all measures
are finite measures on compact transversals).
So we have $\int_{S_\mu}: H^k_{DR}(S_\mu)\to \RR$.

Let $\underline\RR_\mu$ be the sheaf of measurable functions which are locally
constant on leaves and $L^2(\mu)$-transversally.
A standard Poincar\'e lemma shows that there is a resolution of sheaves
 $$
 \underline\RR_\mu \to \Omega^0_{L^2(\mu)} \to \Omega^1_{L^2(\mu)} \to \ldots \to \Omega^k_{L^2(\mu)} \, .
 $$
So we get a natural isomorphism
 $$
 H^p_{DR}(S_\mu) \cong H^p(S,\underline\RR_\mu) \, .
  $$

\begin{definition} The ergodic dimension of $\mu$ is  $1\leq d(\mu )\leq +\infty $ defined to be
the maximal number of mutually singular non-zero transversal measures $\mu_i$ such that
$$
\mu \geq \sum_{i=1}^d c_i \mu_i \ ,
$$
with $c_i>0$ in any transversal.
\end{definition}

By classical ergodic theory, if the ergodic dimension $d(\mu )$ is finite then $\mu$
is a linear combination of exactly $d(\mu )$ ergodic
transversal measures.

\begin{lemma}
The ergodic dimension of $\mu$ is equal to the dimension of $ H^0_{DR}(S_\mu)$,
  $$
  d(\mu )=\dim_\RR H^0_{DR}(S_\mu) \ .
  $$
In particular $\mu$ is ergodic if and only if $H^0_{DR}(S_\mu)\cong \RR$.
\end{lemma}

\begin{proof}
  Let $d\geq 1$ be finite and $d\leq d(\mu)$. By definition,
  $\mu\geq\sum_{i=1}^{d} c_i \mu_i$, where $\mu_i$ are ergodic measures and
  $c_i>0$. 
  Consider disjoint measurable subsets $S_i\subset S$ which are leaf-saturated, such that
  $S_i$ is of total measure for $\mu_i$ and of zero measure for any $\mu_j$ with $j\not= i$
  in any transversal. Then the characteristic functions $\chi_{S_i}$ 
  are independent in $\Omega^0_{L^2(\mu)}(S)$,
  since $\mu(S_i)>0$ for all $i$: if $f=\sum \lambda_i\chi_{S_i}=0$ then $0=||f||^2=\int_{S_\mu}
  |f|^2 = \sum \lambda_i^2 \mu(S_i) \implies \lambda_i=0$ for all $i$. (Here we fix an
  auxiliary Riemannian metric and consider the daval measures corresponding to the
  transversal measures.)

  Thus $\dim_\RR H^0_{DR}(S_\mu) =\dim_\RR
  H^0(S, \underline \RR_\mu ) \geq d$. This proves the result when $d(\mu )=+\infty$.

  Now assume that $d(\mu)$ is finite. Write $\mu=\sum_{i=1}^{d} c_i \mu_i$, where
  $\mu_i$ are ergodic measures, $d=d(\mu)$. Let $S_i$ be as before. Let us prove that
  the $(\chi_{S_i})$ do generate
  $H^0(S, \underline \RR_\mu )$.

  Let $f$ be any measurable function locally constant on leaves
  and $L^2(\mu)$-transversally.
  Note that $f$ is automatically $L^2(\mu_j)$-transversally.
For a transversal $T$, the real function $x\mapsto \mu_{j,T} (T\cap f^{-1} ((-\infty , x]))$ is a non-decreasing
Heaviside function taking the values $0$ and $1$ (by ergodicity of $\mu_j$). So this function has a jump at some
$x_j\in\RR$ that is independent of the transversal $T$ (as holonomy shows).
Then $F_{x_j}=f^{-1}(x_j)$ has total $\mu_j$-measure.
So $F_{x_j}=S_j$ up to a set of $\mu_j$-measure zero, thus
of $\mu$-measure zero. That means that $f$ is constant along each $S_j$ up to a set of
$\mu$-measure zero. So $f= \sum x_j \chi_{S_j}$ in $L^2(\mu)$
(note that $S-(\cup S_j)$ is of zero $\mu$-measure).
\end{proof}

We review basic properties of the De Rham $L^2$-cohomology:
\begin{enumerate}
 \item There is not cup product, and therefore the $H^*_{DR}(S_\mu)$ are just vector spaces (not rings).
 \item Functoriality. If $f:S_1\to S_2$ is a solenoidal map, then we require that $\mu_2=f_* \mu_1$.
 This means that for any local transversal $T_1$ of $S_1$, $f(T_1)$ is a local transversal of $S_2$ and
 the transported measure $f_* \mu_1$ is a constant multiple of $\mu_2$ on the transversal. Note that
 this is automatic when the solenoids are uniquely ergodic. Then for any form $\omega$ which is
 $L^2(\mu_2)$-transversally smooth we have that  $f^*\omega$ is $L^2(\mu_1)$-transversally smooth.

 \item Mayer-Vietoris. It holds exactly as in subsection \ref{subsec:DrRham}.
 \item Poincar\'{e} duality. We shall see that it holds for
 the reduced $L^2$-cohomology for oriented ergodic solenoids (see corollary \ref{cor:PD}).
\end{enumerate}

\section{Bundles over solenoids}\label{sec:bundles-solenoids}

Let $S$ be a $k$-solenoid. A vector bundle of rank $n$ over $S$ consists
of a $(k+n)$-solenoid $E$ and a projection map $\pi:E\to S$ satisfying
the following condition: there
is an open covering $U_\alpha$ for $S$, and solenoid isomorphisms $\psi_\a:
E_\alpha=\pi^{-1}(U_\alpha) \stackrel{\cong}{\to} U_\a\x \RR^n=D^k\x K(U_\a)\x \RR^n$, such that
$\pi=pr_1\circ\psi_\a$, where $pr_1: U_\a\x \RR^n\to U_\a$ denotes the projection, and the transition functions
 $$
 \psi_\a\circ\psi_\b^{-1}: (U_\beta\cap U_\a )\x \RR^n \to
 (U_\beta\cap U_\a )\x \RR^n
 $$
are of the form $(x,y,v)\mapsto (x, y,g_{\a\b}(x,y)(v))$, where $g_{\a\b}$ is a
$C^{\infty,0}$-smooth function
from $U_\a\cap U_\b$ to $\GL(n)$.

Some points are easy to check:
\begin{enumerate}
 \item The usual constructions of vector bundles remain valid here:
 direct sums, tensor products, symmetric and anti-symmetric
 products. Also there are notions of sub-bundle and of quotient
 bundle.
 \item A section of a bundle $\pi:E\to S$ is a map $s:S\to E$ such that $\pi\circ s=Id$.
 We denote the space of sections as $\Gamma(E)$. By definition these are maps of
class $C^{\infty,0}$.
 \item If $S_\mu$ is a measured solenoid, and $E\to S$ is a vector bundle, then
 we have the notion of sections which are $L^2(\mu)$-transversally smooth.
 Locally, in a chart $E_\a=D^k\x K(U)  \x \RR^n \to U_\a=D^k\x K(U)$, the section is
 written $s(x,y)=(x,y,v(x,y))$. We require that $v$ is $C^\infty$ on $x$ and
 $L^2(\mu)$ on $y$. This does not depend on the chosen trivialization.

 \item If $f:S_1\to S_2$ is a solenoid map, and $\pi:E\to S_2$ is a vector bundle,
then the pull-back $f^*E=\{(p,v)\in S_1\x E \ | \ f(p)=\pi(v)\}$
is naturally a vector bundle over $S_1$.
 \item The tangent bundle $TS$ of $S$ is an example of vector bundle.
We have bundles of $(p,q)$-tensors $TS^{\ox p}\ox
(TS^*)^{\ox q}$ on any solenoid $S$. In particular, we have
bundles of $p$-forms (anti-symmetric contravariant tensors)
$\bigwedge^p T^*S$. Its sections are the $p$-forms $\Omega^p (S)$.
 \item A metric on a bundle $E$ is a section of $\Sym^2(E^*)$ which
is positive definite at every point. A metric on $S$ is a metric
on the tangent bundle. An orientation of a bundle $E$ is a continuous choice of
orientation for each of the fibers of $E$. An orientation of $S$
is an orientation of its tangent bundle.
\end{enumerate}

We define $\Omega^p(E)=\Gamma(\bigwedge^p  T^*S \ox E)$.
A connection on a vector bundle $E\to S$ is a map
 $$
 \nabla:\Gamma(E)\to \Omega^1(E),
 $$
such that $\nabla (f\cdot s)=
f\nabla s + df \wedge s$.
Consider a local trivialization in a flow-box $U_\a$ with coordinates $(x,y)$. Then
$\nabla|_{U_\a}=d+a_\a$, where
$a_\a \in \Omega^1(U_\a, \End E)$. Under a change of trivialization $g_{\a\b}$, for two trivializing
open subsets $U_\a, U_\b$, we have the usual formula $a_\b=
g_{\a\b}^{-1} a_\a g_{\a\b} + g_{\a\b}^{-1} d g_{\a\b}$.

A partition of unity argument proves that there are always connections on
a vector bundle $E\to S$.
The space of connections is an affine space over $\Omega^1(\End E)$.

Given a connection $\nabla$ on $E$, there is a unique map
$d_\nabla: \Omega^p(E)\to \Omega^{p+1}(E)$, $p\geq 0$, such that
$d_\nabla s= \nabla s$ for $s\in \Gamma(E)$, and
 $d_\nabla(\alpha\wedge \beta)= d \alpha\wedge \b+(-1)^p  \a\wedge
 d_\nabla \b$, for $\a\in\Omega^p(S)$, $\b\in \Omega^q(E)$.
It is easy to see that $\hat{F}_\nabla:\Gamma(E)\to \Omega^2(E)$,
given by $\hat{F}_\nabla (s)= d_\nabla d_\nabla s$, has a tensorial character
(i.e., it is linear on functions). Therefore there is a $F_\nabla
\in \Omega^2(\End E)$, called \emph{curvature} of $\nabla$, such
that $\hat{F}_\nabla (s)=F_\nabla \cdot s$. Locally on a trivialization
$U_\a$, we have the formula $F_\nabla =da_\a +
a_\a\wedge a_\a$.

Given connections on vector bundles, there are induced
connections on associated bundles (dual bundle, tensor product,
direct sum, symmetric product, pull-back under a solenoid map, etc.).
This follows in a straightforward way from the standard theory.
In particular, if $l\inc S$ is a leaf of a solenoid $S$, then
we can perform
the pull-back of the bundle and connection to the leaf, which consists
on restricting them to $l$. This gives
a bundle and connection of a complete $k$-dimensional manifold.
Also, if $f:S\to M$ is an immersion of a solenoid in a smooth $n$-manifold,
and $E\to M$ is a bundle with connection, then the pull-back construction
produces a bundle with connection on $S$.

Consider a vector bundle $E\to S$ endowed with a
metric. We say that a connection $\nabla$ is compatible with the metric if it
satisfies
  $$
  d\la s,t \ra = \la \nabla s,t\ra + \la s,\nabla t\ra\, .
  $$
In the particular case of the tangent bundle $TS$ of a Riemannian solenoid $S$,
we have the Levi-Civita connection $\nabla^{LC}$, which is the unique connection
compatible with the metric and with torsion $T_{\nabla}(X,Y)=\nabla_X Y- \nabla_Y X=0$.
This is the Levi-Civita connection on each leaf, and the transversal continuity
follows easily.

\section{Chern classes} \label{sec:chern-classes}

We can also define a complex vector bundle over a solenoid, by
using $\CC^n$ as fiber, and taking the transition functions with
values in $\GL(n,\CC)$. An hermitian metric on a complex vector bundle is a
positive definite hermitian form in each fiber with
smoothness of type $C^{\infty, 0}$ on any local trivialization.

Let $E\to S$ be a complex vector bundle over a solenoid of rank $n$. Put a hermitian
structure on $E$, and consider any hermitian
connection $\nabla$ on $E$. Then the curvature $F_\nabla$ is a $2$-form
with values in $\End E$, i.e. $F_\nabla \in \Omega^2(\End E)$.
The Bianchi identity says
 $$
 d_\nabla F_\nabla =0\, .
 $$
This holds leaf-wise, so it holds on the solenoid.

Consider the elementary functions: $\Tr_i: M_{r\times r} \to \CC$, given  by
$\Tr_i(A)=\Tr(\bigwedge^i A)$. Then the Chern classes are
 $$
  c_i(E)=\Big[\Tr_i \Big(\frac{\sqrt{-1}}{2\pi} F_\nabla\Big)\Big] \in H^{2i}_{DR}(S)\, .
 $$
These classes are well defined (since the forms inside are closed, which
again follows by working on leaves) and
do not depend on the connection (different connections give forms
differing by exact forms), see \cite[Chapter III]{Wells}.

We have some facts:
\begin{enumerate}
 \item If $M$ is a manifold, we recover the usual Chern classes.
 \item If $f:S_1\to S_2$ is a solenoid map, then $f^*c_i(E)=c_i(f^*E)$. In particular,
\begin{itemize}
 \item If $f:S\to M$ is an immersion of a solenoid in a manifold and
 $E|_S=f^*E$, then $c_i(E|_S)=f^*c_i(E)$.
 \item If $j:l\to S$ is the inclusion of a leaf, then $c_i(E|_l)= j^* c_i(E)$.
\end{itemize}
\end{enumerate}

\begin{question}
Are the Chern classes defined as elements
in $H^{2i}(S,\underline\ZZ)$?
\end{question}

This question has an affirmative answer for the case of line bundles (complex vector bundles of rank $1$).
A line bundle $L\to S$ is given by its transition functions $g_{\alpha\beta}:
U_{\alpha}\cap U_{\beta} \to S^1$ which are of class $C^{\infty,0}$. Therefore the line bundles on
$S$ are parametrized by $H^1(S,C^{\infty,0}(S^1))$, where $C^{\infty,0}(S^1)$ is the sheaf which
assigns $U\mapsto C^{\infty,0}(U, S^1)$.

Note that there is an exact sequence of sheaves:
 $$
  0 \to \underline\ZZ \to C^{\infty,0}(\RR) \to C^{\infty,0}(S^1) \to 0 \, .
  $$
Here the sheaf $\underline\ZZ$ is the locally constant sheaf.
As $C^{\infty,0}(\RR)$ is a fine sheaf (it has partitions of unity), it is
acyclic. So the map
 $$
 \delta: H^1(S,C^{\infty,0}(S^1)) \to H^2(S, \underline\ZZ)
 $$
is an isomorphism. There is a natural map $\alpha:H^2(S,\underline\ZZ)\to H^2(S,\underline\RR_c)$.
It is easy to see that $\alpha(\delta ([L]))=c_1(L)$, by using the transition functions $g_{\a\b}$
to construct a suitable connection on $L$ with which we compute the curvature (as in the case of
manifolds, see \cite{Wells}).

\section{Hodge theory} \label{sec:hodge-theory}

\subsection{Sobolev norms}

Let $S_\mu$ be a compact Riemannian $k$-solenoid which is oriented and endowed with
a transversal measure. We denote the associated (finite) daval measure also by $\mu$.
Now consider a vector bundle $E\to S$ and endow it with a
metric. The space of sections of class $C^{\infty,0}$ is denoted $\Gamma(S,E)$. The space
of $L^2(\mu)$-transversally smooth sections (sections of class $C^\infty$ along leaves and
$L^2$ in the transversal directions) is denoted by $\Gamma_{L^2(\mu)}(S,E)$.

Now let us introduce suitable completions of these spaces of sections.
Fix a connection $\nabla$ for $E$ and the Levi-Civita connection for
$TS$. There is an $L^2$-norm on sections of $E$, given by
 $$
 (s,t)_E= \int_{S} \la s,t\ra \, d\mu \, .
 $$
We can complete the spaces of sections to obtain spaces
of $L^2$-sections $L^2_\mu (S,E)$.
We consider also Sobolev norms $W^{l,2}$ as follows.
Take $s$ a section of $E$. Then we set
 $$
 ||s||_{W^{l,2}_\mu}^2=\int_{S}
 \sum_{i=0}^l |\nabla^i s|^2 \, d\mu \, .
 $$
Completing with respect to this norm gives a
Hilbert space consisting of sections with regularity $W^{l,2}$ on leaves
and $L^2(\mu)$-transversally, denoted $W^{l,2}_\mu(S,E)$. These spaces
do not depend on the choice of metrics and connections.

For future use, we also introduce the norms $C^r_\mu$, which give spaces
of sections with $C^r$-regularity on leaves and $L^2(\mu)$-transversally.
Take $s$ a section of $E$. Assume
it has support in a flow box $U=D^k\x K(U)$, and assume that $E$ has been trivialized
by an orthonormal frame. Then
 $$
 ||s||_{C^r_\mu}^2=\int_{K(U)}
 ||s( \cdot, y) ||_{C^r}^2 \, d\mu_{K(U)}(y) \, .
 $$
These norms are patched (via partitions of unity,
in a non-canonical way) to get a norm on the
spaces of sections on the whole solenoid. The topology defined by this norm is independent of the partition of unity.
The spaces of sections are denoted $C^r_\mu(S,E)$.
Note that $\bigcap_{r\geq 0} C^r_\mu(S,E)={L^2_\mu}(S,E)$.

We can define the norm $W^{l,2}_\mu$ by using Fourier transforms.
For this we have to restrict to a flow-box $U=D^k\x K(U)$.
We Fourier-transform the section $s(x,y)$ in the leaf-wise
directions, to get $\hat{s}(\xi,y)$, and then take the integral
 $$
 \int_{K(U)} \left(\int  (1+|\xi|^2)^l |\hat{s}(\xi,y)|^2 d\xi \right) d\mu_{K(U)}(y).
 $$

\begin{proposition}[Sobolev]
  $W^{s,2}_\mu(S,E) \subset C^p_{\mu}(S,E)$, for $s>[k/2] + p + 1$.
\end{proposition}

This is similar to Proposition 1.1 in Chapter IV of \cite{Wells}.
The proof carries over to the solenoid situation verbatim.
As a consequence,
  $$
  \bigcap_{r\geq 0} W^{r,2}_\mu(S,E)= \Gamma_{L^2(\mu)}(S,E)\,.
  $$

\subsection{Pseudodifferential operators}
Let $E,F$ be two vector bundles over $S$ of ranks $n,m$ respectively. A differential operator $L$ of
order $l$ is an operator
 $$
 L:\Gamma(S,E) \to \Gamma(S,F)
 $$
which locally on a flow-box $U=D^k\x K(U)$ is of the form
 $$
 L (s)= \sum_{|\alpha|\leq l} A_\alpha (x,y) D^\alpha s \, ,
 $$
where $A_\alpha$ are $(n\x m)$-matrices of functions (with regularity $C^{\infty,0}$)
and $\alpha=(\alpha_1,\ldots, \alpha_k)$ is a multi-index, with
$|\alpha|=\sum \alpha_i$, $D^\alpha=\frac{\bd^{|\alpha|}}{\bd^{\alpha_1} x_1
\ldots \bd^{\alpha_k} x_k}$.
Note that a differential operator gives rise to differential operators on each leaf.
Moreover, $L$ extends to
 $$
 L:{W^{p,2}_\mu} (S,E) \to {W^{p-l,2}_\mu} (S,F)\, .
 $$
The usual properties, like the existence of adjoints, extend to this setting.

The symbol of a differential operator on a solenoid is defined in the same fashion as
for the case of manifolds, and coincides with the symbol of the differential operator
on the leaves. We recall that the symbol
$\sigma_l(L) \in \Hom(\pi^*E, \pi^*F)$, $\pi:TS\to S$, has the form
 $$
 \sigma_l (L)(x,y, v)= \sum_{|\alpha|=l} A_\alpha(x,y) v_1^{\alpha_1}\ldots  v_k^{\alpha_k}\, .
 $$
The properties of the symbol map, such as the rule of the symbol
of the composition of differential operators, or the symbol of the adjoint, hold here.
This is just the fact that they can be done leaf-wise, and the continuous transversality
is easy to check.

Differential operators can be generalized to pseudodifferential operators as in the
case of manifolds. A pseudodifferential operator of order $l$ on a flow-box
$U=D^k\x K(U)$ is an operator
  $$
  L(p): \Gamma_c(U,E) \to \Gamma(U,F)
  $$
which sends a (compactly supported) section $s(x,y)$ to
 $$
 L(p) s (x,y)= \int p(x,\xi, y) \hat{s}(\xi,y) e^{i \la x,\xi \ra} d \xi \, ,
 $$
where $\hat{s}(\xi,y)$ is the (leaf-wise) Fourier transform, and $p(x,\xi, y)$
is a function defined in $D^k \x \RR^k \x K(U)$, smooth on $x$ and $\xi$, continuous on $y$,
and satisfying:
\begin{itemize}
 \item   $|D^\beta_x D^\alpha_\xi p(x,\xi,y)| \leq C_{\a\b\, l} (1 +|\xi|)^{l-|\a|}$, for constants $C_{\a\b\, l}$,
\item the limit $\sigma_l(p)(x,\xi,y)= \lim_{\lambda \to \infty} \frac{p(x,\lambda \xi,y)}{\lambda^l}$
 exists,
 \item $p(x,\xi,y) - \sigma_l(p) (x,\xi,y)$ should be of order $\leq l-1$ for $|\xi|\geq 1$.
\end{itemize}

A pseudodifferential operator of order $l$ on $S$ is an operator $L:\Gamma(S,E) \to \Gamma(S,F)$
which is locally of the form
$L(p_U)$ for some $p_U$ as above. The symbol of $L$ is $\sigma_l(L)=\sigma_l(p_U)$ for a local representative
$L|_U=L(p_U)$. This symbol is well-defined and independent of choices, which is a delicate point but it
is analogous to the case of manifolds (see \cite{Wells}).
The usual properties of the symbol map (composition, adjoint) hold here.

A pseudodifferential operator of order $l$ is an operator of order $l$, i.e.,
it extends as a continuous map to
 $$
 L :{W^{p,2}_\mu} (S,E) \to {W^{p-l,2}_\mu} (S,E)\, .
 $$
This is done as in Theorem 3.4 of \cite[Ch. IV]{Wells}, by noting that
$|| L(p) s ( \cdot, y)||_{W^{p-l,2}_\mu} \leq C || s(\cdot, y)||_{W^{p,2}_\mu}$,
where $C$ is a constant depending on $C_{\a\b\, l}$.

The key of the theory is the fact that we can construct a pseudodifferential operator
given a symbol $\sigma_l(L)$.
\begin{proposition}\label{prop:symbol}
  Let $S$ be a compact solenoid. Then
  there is an exact sequence $0 \to \mathrm{OP}_{l-1}(E,F) \to
  \mathrm{PDiff}_l(E,F) \to \mathrm{Symb}_l(E,F)\to 0$, where
  $\mathrm{OP}_{l-1}(E,F)$ is the space of operators of order $l-1$,
 $\mathrm{PDiff}_l(E,F)$ the space of pseudodifferential operators of order $l$,
 and $\mathrm{Symb}_l(E,F)$ the space of symbols of order $l$.
\end{proposition}

\subsection{Elliptic operator theory for solenoids}

We say that a pseudodifferential operator $L:E\to F$ of order $l$ is elliptic if the symbol
$\sigma_l(L)$ satisfies that $\sigma_l(L)(x,v):E_x\to F_x$ is an isomorphism for each
$x\in S$, $v\in T_xS$, $v\neq 0$.

\begin{theorem} \label{thm:ll}
 Let $L$ be an elliptic pseudodifferential operator of order $l$. Then there exists a
 pseudo-inverse, a pseudodifferential operator $\tilde{L}$ of order $-l$ such that
 $L\circ \tilde{L} = Id+K_1$ and $\tilde{L}\circ L = Id+K_2$, where $K_1,K_2$ are operators
 of order $-1$.
\end{theorem}

This is done as in Theorem 4.4 of \cite[Ch.\ IV]{Wells}.
The basic idea is to construct a pseudo-inverse by using Proposition \ref{prop:symbol}.
Note that $K_1, K_2$ are not usually compact operators (this is due to the failure of
the Rellich lemma in our situation), so we will not have finite-dimensionality of the
kernel and cokernel of elliptic operators.

\begin{corollary}
 Let $L$ be an elliptic pseudodifferential operator of order $l$, and let $\cK_{L_s}=\ker
( L: {W^{s,2}_\mu}(S,E)\to {W^{s-l,2}_\mu}(S,F))$. Then $\cK_{L_s}\subset \Gamma_{L^2(\mu)}(S,E)$,
and it is independent of $s$.
\end{corollary}

\begin{proof}
Let $\sigma \in W^{2,s}_\mu(S,E)$ such that $L\sigma =0$. Then $\sigma =(\tilde{L}\circ L- K_2)(\sigma)
=-K_2\sigma \in W^{2,s+1}_\mu(S,E)$. Working inductively, $\sigma\in \cap_{r\geq 0} W^{2,r}_\mu(S,E)=
\Gamma_{L^2(\mu)}(S,E)$. The second assertion is clear.
\end{proof}

An operator $L:\Gamma(E) \to \Gamma(E)$ is called self-adjoint if $L^*=L$.
If $L$ is an elliptic self-adjoint operator, then there is a
 pseudo-inverse $G$ which is self-adjoint (just take the pseudo-inverse
 $\tilde{L}$ provided by  Theorem \ref{thm:ll} and let $G=(\tilde{L} + \tilde{L}^*)/2$).
 Then we have that $L\circ G=G\circ L$, because
  $$
  \la (L\circ G-G\circ L) s,s\ra= \la Gs,Ls\ra - \la Ls, Gs\ra =0\, .
  $$
In particular, $K_1=K_2$ in Theorem \ref{thm:ll}.

For self-adjoint operators, we have the following result

\begin{theorem} \label{thm:G}
 Let $L$ be an elliptic self-adjoint operator of order $l$. Then
 $$
  {W^{s,2}_{\mu}}(S,E)=\ker L\oplus \overline{\im L}\, .
  $$
and an analogous result for $\Gamma_{L^2(\mu)}(S,E)$.
 \end{theorem}

\begin{proof}
 Clearly, $\ker L$ is a closed subspace. Let $s_1\in \ker L$ and $s_2\in \im L$, say
 $s_2=L(t)$. From this we have $\la s_1,s_2\ra=\la s_1, L(t)\ra= \la L(s_1), t\ra =0$. The
 decomposition $\ker L\oplus\overline{\im L}$ is therefore orthogonal.

 It remains to prove that $(\ker L)^\perp \subset \overline{\im L}$.
 Equivalently, $(\im L)^\perp \subset \ker L$. Let $s\in (\im L)^\perp$.
 Then $\la L(s), s_2\ra= \la s, L(s_2)\ra=0$,
 for all $s_2$. Hence $L(s)=0$. 
\end{proof}


A complex of differential operators is a sequence
 $$
 \Gamma(E_0) \stackrel{L_0}{\too} \Gamma(E_1) \stackrel{L_1}{\too}  \ldots \stackrel{L_{m-1}}{\too}\Gamma(E_m) \, ,
 $$
where $E_i$ are vector bundles, and $L_i$ are differential operators such that
$L_i\circ L_{i-1}=0$. The complex is called elliptic if the sequence of symbols
 $$
 \pi^* E_0 \stackrel{\sigma(L_0)}{\too} \pi^* E_1
 \stackrel{\sigma(L_1)}{\too}  \ldots \stackrel{\sigma(L_{m-1})}{\too}\pi^*E_m \, ,
 $$
is exact for each $v\neq 0$. We define the cohomology of the complex as
 $$
 H^q(S,E) = \frac{\ker (L_q:\Gamma(E_q) \to \Gamma(E_{q+1}))}{\im (L_{q-1}:
\Gamma(E_{q-1}) \to \Gamma(E_{q}))} \, ,
 $$
and the $L^2$-cohomology by
 $$
 H^q(S_\mu,E) = \frac{\ker (L_q:\Gamma_{L^2(\mu)}(E_q) \to \Gamma_{L^2(\mu)}(E_{q+1}))}{\im
 (L_{q-1}: \Gamma_{L^2(\mu)}(E_{q-1}) \to \Gamma_{L^2(\mu)}(E_{q}))}\, .
 $$
The \emph{reduced} $L^2$-cohomology is
 $$
 \bar H^q(S_\mu,E) = \frac{\ker L_q}{\, \overline{\im L_{q-1}}\, } \, .
 $$
This is the group $H^q(S_\mu,E)$ quotiented by the closure of $\{0\}$, making it
a Hausdorff space.

We construct the Laplacian operators of the elliptic complex as follows:
 $$
 \Delta_j=L_j^*L_j + L_{j-1}L_{j-1}^* : \Gamma_{L^2(\mu)}(E_j)\to \Gamma_{L^2(\mu)}(E_j)\, .
 $$
These are self-adjoint elliptic operators. There is an associated
operator $G$ given by Theorem \ref{thm:G}. Denote
  $$
  \cH^j(E)=\ker \Delta_j\, .
  $$
And note that $\Delta_j s=0$ if and only if $L_js=0$ and $L_j^*s=0$.
We remove the subindex $j$ from now on.

\begin{theorem} We have the following:
\begin{enumerate}
 \item $\overline{\im \Delta} =\overline{\im L}\oplus \overline{\im L^*}$, and it is
 an orthogonal decomposition.
 \item $\Gamma_{L^2(\mu)}(S,E_j)=\cH^j(E)\oplus \overline{\im L}\oplus \overline{\im L^*}$.
 \item There is a canonical isomorphism $\cH^j(E)\cong \bar H^j (S_\mu, E)$.
 \end{enumerate}
\end{theorem}

\begin{proof}
(1) The inclusion $\subset$ is clear. Note that the decomposition is orthogonal:
 $\la L(s_1), L^*(s_2)\ra =\la LL (s_1),s_2\ra =0$, for all $s_1,s_2$.

For the reverse inclusion, let us check that $\im L\subset \overline{\im \Delta}$
(the other inclusion is similar). This is equivalent to proving that $\ker \Delta\subset
(\im L)^\perp$. But $\ker \Delta = \ker L\cap \ker L^* \subset \ker L^* \subset (\im L)^\perp$,
so the result follows.


(2) follows from (1) and Theorem \ref{thm:G}.

(2) Consider the map $\cH^j(E)\to \bar H^j (S_\mu, E)$. This is well defined because
if $\Delta s=0$, $s\in \Gamma_{L^2(\mu)}(S,E_j)$ then
$0=\la s,\Delta s\ra= \la s, (LL^*+L^*L)s\ra = \la L^*s, L^*s\ra +
\la Ls, Ls\ra \implies Ls=L^*s=0$.

It is injective. If $s= L t$ and $\Delta s=0$ then $0=L^*s=L^*Lt$, so $||Lt ||^2=\la L^*Lt,t\ra=0$,
so $s=Lt=0$.

It is surjective. Take a class $[s]\in \bar H^j (S_\mu, E)$. Decompose $s=s_1+s_2+s_3$, where $s_1\in
\cH^j(E)$, $s_2 \in \overline{\im L}$ and $s_3\in \overline{\im L^*}$. Now $0=Ls=Ls_3$, so
$s_3 \perp \im L^*$. Therefore $s_3=0$. So $s=s_1+s_2$ and $[s]=[s_1]$
in $\bar H^j(S_\mu,E)$, where $s_1 \in \cH^j(E)$.
\end{proof}

\subsection{Harmonic theory}
The Riemannian metric and the orientation 
give rise to a natural volume form along leaves $\vol
\in \Omega^k(S)$. The usual Hodge-$*$ operator (see \cite{Wells}) can be defined for
forms on $S$, actually, it is the $*$ operator on leaves. This operator
$*:\Omega^p(S) \to \Omega^{k-p}(S)$ is defined by
  $$
   \alpha \wedge * \beta =   ( \alpha, \beta)  \, \vol\, ,
  $$
for $\alpha,\beta \in \Omega^p(S)$, where $(\cdot, \cdot)$ is the point-wise metric induced on
forms. Note that $*$ extends to $*:\Omega^p_{L_\mu^2}(S) \to \Omega^{k-p}_{L_\mu^2}
(S)$, since it is leaf-wise isometric. Note that $\vol= * 1$.

It is easy to check that  $d^*=\pm * d *$. 
The Laplacian is defined as $\Delta=dd^* +d^* d$. Note that if $\Delta s=0$ then
$(s, \Delta s)=(s, dd^*s)+ (s, d^*ds)= (d^*s,d^*s)+(ds,ds)=||d^*s||^2+ ||ds||^2$. So
$d^*s=0$ and $ds=0$. We define the space of harmonic forms:
  $$
   \cK^j(S_\mu)=\cH_\Delta(\wedge^j T^*S)\, .
   $$
Then the theory of elliptic operators says the following

\begin{theorem} We have:
  \begin{itemize}
   \item The space of harmonic sections $\cK^j(S_\mu)\subset \Omega^j_{L^2(\mu)}(S)$.
   \item There is a natural isomorphism $\bar H^j_{DR}(S_\mu) \cong \cK^j(S_\mu)$.
     \end{itemize}
\end{theorem}

\begin{corollary} \label{cor:PD}
Poincar\'e duality:
 $$
 * :\cK^p(S_\mu)\to \cK^{k-p}(S_\mu)
 $$
is an isomorphism. 

If $S$ is ergodic, then $H^0_{DR}(S_\mu)\cong H^k_{DR}(S_\mu)\cong \RR$ (with the isomorphism given by
integration $\int_{S_\mu}$. Therefore
 $$
 \int_{S_\mu} : \bar H^p_{DR}(S_\mu) \ox \bar H^{k-p}_{DR}(S_\mu) \to \RR
 $$
is a perfect pairing.
\end{corollary}


In general, the spaces $\cK^p(S_\mu)$ are not in general finite dimensional. For instance, take
a solenoid which is a fibration, i.e., $S$ is a compact $(n+k)$-manifold
such that there is a submersion $\pi:S\to B$ onto an $n$-dimensional manifold,
and the transversal measure is induced by a measure $\mu$ on $B$.
Then we have a fiber bundle $\HH^p\to B$ such that $\HH^p_y=H^p(F_y)$, $F_y=\pi^{-1}(y)$, $y\in B$.
Then $\cK^p(S_\mu)\cong L^2_\mu(\HH^p)$.

Nonetheless, we propose the following:

\begin{question} \label{conj}
If $S_\mu$ is an ergodic solenoid with controlled
growth, are the spaces $\cK^p(S_\mu)$ of finite dimension?
\end{question}

A controlled solenoid has transversal measures and is defined in \cite{MPM2}.
For such a solenoid, any leaf has an
exhaustion $K_n$ by compact sets such that for any flow-box in a finite covering, the
number of local leaves which intersect partially $K_n$ is negligible.
In this case, the normalized measure supported on $K_n$ converges to a daval measure
(giving rise to a transversal measure, as usual).
If the solenoid $S_\mu$ is ergodic, then for almost all $\mu$-leaves, this limit coincides with
$\mu$. So in this situation, we can understand the behaviour of harmonic forms by restricting
to a leaf, and then study the spaces of harmonic forms on this leaf (which is a complete
manifold with a quasi-periodic behaviour).

The example in section \ref{sec:example} shows that the question has a negative
answer if we ask instead for $H_{DR}^p(S_\mu)$.

\section{Example: Kronecker solenoids} \label{sec:example}

Let us develop the example of Kronecker solenoids, i.e. the flat torus with a linear foliation. Let $\TT^n=\RR^n/\ZZ^n$ with
a flat euclidean metric on $\TT^n$. Consider a foliation given by a $k$-dimensional
linear subspace $W\subset \RR^n$. That is, fix an orthonormal basis $w_1,\ldots, w_n$,
so that $W=\langle w_1,\ldots, w_k\rangle$.
Consider the solenoid $S$ whose leaves
are the images of $k$-planes of $\RR^n$ parallel to $W$ under $\pi:\RR^n\to \TT^n$.

The solenoid is minimal with all leaves dense in $\TT^n$ if there is no
hyperplane $H=\sum m_ix_i=0$ with $m_i\in \ZZ$ so that $W\subset H$. Equivalently,
for any integer vector $m \in \ZZ^n$, we have that $\langle m, w_j \rangle \not=0$
for some $j=1,\ldots, k$. Equivalently,
 \begin{equation}\label{eqn:last-condition}
 \sum \langle m, w_j \rangle^2 \not=0 \ .
 \end{equation}
Notice that all leaves have the same topological type and are simply connected and 
diffeomorphic to $\RR^k$ if and only if $W\cap \ZZ^n =\{ 0\}$. We shall suppose that
the solenoid is minimal, but we allow $W\cap \ZZ^n \not=\{ 0\}$.

In this situation, we have affine transversals $T\subset W^\perp=\langle w_{k+1},\ldots, w_n \rangle $ and the holonomy are irrational
translations, therefore by Haar theorem the unique transversal measure is the Lebesgue measure
on the transversal. Thus $S$ is uniquely ergodic, and the unique daval measure
is the Haar measure on $\TT^n$, image of the Lebesgue measure on $\RR^n$. .

We can characterize the space of functions $C^\infty$
on leaves and $L^2$-transversally by their Fourier expansion. All
such functions are in particular $L^2$ on the torus,
so they have a Fourier expansion ( $x=(x_1,\ldots, x_n)\in \TT^n$, $m=(m_1,\ldots , m_n)\in \ZZ^n$)
 $$
 f(x)= \sum_{m\in \ZZ^n} a_{m} \ e^{2\pi i \langle m , x\rangle } \, ,
 $$
where $\sum |a_{m}|^2< \infty$.

To be smooth on leaves is equivalent to be $W^{l,2}$ on leaves for all $l$. Let us take a
derivative along some $w_j$, $j=1,\ldots, k$.
$$
 D_{w_j} f= \sum_{m\in \ZZ^n} \langle m , w_j \rangle \, a_m \  e^{2\pi i \langle m , x\rangle}
 $$
So the condition  $C^\infty$ on leaves and $L^2$-transversally is
 $$
 \sum_{m\in \ZZ^n} | \prod_{j=1}^k \langle m , w_j \rangle ^{r_j}\,  a_{m}|^2< \infty
 $$
for all $r_1,\ldots, r_k \geq 0$.

Note that the Laplacian on functions is the usual Laplacian on the leaves direction, so that
 $$
 \Delta f= -\sum_{j=1}^k D_{w_j}^2 f = -\sum_{m\in \ZZ^n}  \left (\sum_{j=1}^k  \langle m ,
  w_j \rangle^2 \right )
 a_{m} \ e^{2\pi i \langle m , x\rangle} \, .
 $$
So $\Delta f=0 \iff f=$ constant, using (\ref{eqn:last-condition}).

Now consider $p$-forms, $0\leq p \leq k$. A $p$-form $\omega=\sum_{|J|=p} f_J \, d\varpi_J$, where
$\varpi_1,\ldots, \varpi_k$ are coordinates on the leaves, dual to the basis $w_1,\ldots, w_k$,
and $J=(j_1,\ldots, j_p)$, $1\leq j_1<\ldots <j_p\leq k$, $d\varpi_J=d\varpi_{j_1}\wedge \ldots
\wedge d\varpi_{j_p}$. For the flat metric, $\Delta  \omega=\sum_{|J|=p} \Delta f_J \, d\varpi_J$.
So $\Delta \omega=0 \iff f_J=$ constant.

In conclusion,
 $$
 \cK^p(S_\mu)=\bigwedge\nolimits^p \,  W^* \, ,
 $$
which is of dimension $\binom{k}{p}$.

Note that the $*$-Hodge operator satisfies $*\varpi_J=\pm\varpi_{J^c}$, where $J^c$ is the complement
$\{1,\ldots, k\}-J$. So $*:\cK^p(S_\mu)\to \cK^{k-p}(S_\mu)$ is the natural isomorphism
$\bigwedge^p W^*\cong \bigwedge^{k-p} W^*$.

\noindent \textbf{The case of the Kronecker $1$-solenoid.}
Let us see the particular case of the Kronecker
$1$-solenoid with $w_1=\alpha=(\alpha_1 ,\ldots , \alpha_n)$ and
$$
\dim_\QQ \la \alpha_1 , \ldots ,\alpha_n \ra=n \ .
$$
As above, $\cK^0(S_\mu)=\RR$ and $\cK^1(S_\mu)=\RR$.

Let us compute $H^0_{DR}(S_\mu)$ and $H^1_{DR}(S_\mu)$.
For a function $f=\sum a_{m} e^{2\pi i \langle m, x\rangle}$,
 $$
 df=\sum \langle m, \alpha \rangle \, a_{m} \ e^{2\pi i \langle m, x\rangle}\, .
 $$
So $df=0 \iff f=$ constant. Then $H^0_{DR}(S_\mu)=\RR$.

Now a $1$-form $\omega= g\, d\varpi_1$, $g=\sum b_{m} \ e^{2\pi i \langle m, x\rangle}$, is exact if and only if
 $$
 g=df=\sum \langle m, \alpha \rangle \, a_{m} \  e^{2\pi i \langle m, x\rangle},
 $$
i.e.
 $$
 a_{m}= \frac{b_{m}}{\langle m, \alpha \rangle }\, , \qquad m\in \ZZ^n\, ,
 $$
should be in $L^2$. So
 \begin{equation}\label{eqn:am}
 H^1_{DR}(S_\mu)=\frac{\{ (b_{m}) | \sum |\langle m, \alpha \rangle^r b_{m}|^2 <\infty, \forall r\geq 0\} }{
\{ (b_{m}) | \sum |\langle m, \alpha \rangle^r b_{m}|^2 <\infty, \forall r\geq -1\}} \, .
 \end{equation}

This space is always infinite dimensional. We can construct an uncountable basis by using
Minkowski's diophantine approximation theorem: there are infinitely many integer vectors $m\in \ZZ^n$
such that
$$
|\langle m, \alpha \rangle | < \frac{1}{||m||} \ .
$$
Then we can choose an uncountable number of {\it disjoint} sequences $(m_k)\subset \ZZ^n$ such that $\langle m_k, \alpha \rangle \to 0$
when $k\to +\infty$.
For each such sequence take a Fourier series with support on it (that is $b_m=0$ if $m\not=m_k$) such that
$\sum_k b_{m_k}^2 < +\infty $ but
$$
\sum_k \left (\frac{b_{m_k}}{\langle m_k, \alpha \rangle }\right )^2=+\infty \ .
$$
We can take for example $b_{m_k}=\langle m_k, \alpha \rangle$. Then clearly $(b_{m_k})$ lies in (\ref{eqn:am}).
In this way we get uncountably many independent elements in (\ref{eqn:am}).

\medskip

As an aside, note that the De Rham cohomology $H^1_{DR}(S)$ is also of infinite dimension
(as it is proved in \cite{Macias} by sheaf theoretic methods). However, if we use forms with regularity
$C^{\infty,\infty}$, using the criterion characterising $C^\infty$ as those with polynomially decaying Fourier coefficients,
we can compute that
 $$
  H^1(S, \RR_\infty) \cong \frac{\{ (b_{m}) \, | \,  \,  ||b_m|| =O ((1+||m||)^{-r}), \forall r>0\}}{
 \{ (b_{m}) \, | \,  \, ||\langle m, \alpha \rangle^{-1} b_m|| =O ((1+||m||)^{-r}), \forall r>0 \}} \,
 $$
(here $\RR_\infty$ is the sheaf of functions locally constant on leaves and $C^\infty$ transversally). For
a vector $\alpha\in DC(\gamma, \tau)$ satisfying a diophantine condition, we have that there exist $\tau >0$ and $\gamma >0$ such that
for all $m\in \ZZ^n$ we have
$$
|\langle m, \alpha \rangle | \geq \frac{\gamma}{||m||^{n+\tau}} \ .
$$
(The subset of diophantine vectors $DC(\tau)=\bigcup_{\gamma>0}
DC(\gamma, \tau )\subset \RR^n$ is of full Lebesgue measure.)
Then it follows that when $\alpha$ is diophantine, 
$H^1(S, \RR_\infty)$ is one dimensional, and otherwise, when
$\alpha$ is Liouville, it is infinite dimensional (\cite{Reinhart},\cite{Heitsch}).
To impose such type of transversal regularity in this problem does not seem to be natural.

\end{document}